\documentclass[12pt]{amsart}

\usepackage{amsmath,amssymb,amsthm,mathtools,mathrsfs}
\usepackage[hidelinks]{hyperref}

\hypersetup{
  pdftitle={Maximal Unramified p-Extensions with Prescribed Galois Groups: A Quantitative Refinement of Ozaki's Theorem},
  pdfauthor={Kwang-Seob Kim},
  pdfkeywords={p-class field tower, unramified extension, Frattini subgroup,
  degree bound, ray class field, Tate sequence, inverse Galois problem}
}

\calclayout
\numberwithin{equation}{section}

\newtheorem{theorem}{Theorem}[section]
\newtheorem{proposition}[theorem]{Proposition}
\newtheorem{lemma}[theorem]{Lemma}
\newtheorem{corollary}[theorem]{Corollary}

\theoremstyle{definition}

\theoremstyle{remark}
\newtheorem{remark}[theorem]{Remark}

\newcommand{\Q}{\mathbb Q}
\newcommand{\Z}{\mathbb Z}
\newcommand{\Fp}{\mathbb F_p}
\newcommand{\Gal}{\operatorname{Gal}}
\newcommand{\Cl}{\operatorname{Cl}}
\newcommand{\Frob}{\operatorname{Frob}}
\newcommand{\Nm}{\operatorname{N}}

\newcommand{\Hom}{\operatorname{Hom}}

\newcommand{\cO}{\mathcal O}
\newcommand{\cC}{\mathscr C}
\newcommand{\cJ}{\mathscr J}
\newcommand{\cM}{\mathfrak m}
\newcommand{\Lp}{L_p}
\newcommand{\PhiG}{\Phi}

\title[A quantitative refinement of Ozaki's theorem]
{Maximal Unramified $p$-Extensions with Prescribed Galois Groups:\\
A Quantitative Refinement of Ozaki's Theorem}
\author{Kwang-Seob Kim}
\address{Department of Mathematics, Chosun University, Gwangju 61452,
South Korea}
\email{kwang12@chosun.ac.kr}
\subjclass[2020]{Primary 11R32; Secondary 11R29, 20D15, 20J06}
\keywords{$p$-class field tower, unramified extension, Frattini subgroup,
degree bound, ray class field, Tate sequence, inverse Galois problem}

\begin{document}

\begin{abstract}
Ozaki proved that every finite $p$-group occurs as the Galois group of
a maximal unramified $p$-extension of a number field.  Hajir, Maire and
Ramakrishna made this theorem effective, obtaining a base-field degree
of order $|G|$.  In this article, we reduce that degree by taking the
Frattini structure of $G$ into account.  More precisely, for an odd
prime $p$ and a finite $p$-group $G$ of order $p^n$, we prove
\[
 \tau_p(G)
 \leq
 p^{\ell_\Phi(G)+
 \left\lceil\log_p\left(\binom{n+2}{2}+1\right)\right\rceil}.
\]
Here $\tau_p(G)$ is the least degree of a number field realizing $G$
as its $p$-class tower group, and $\ell_\Phi(G)$ is the iterated
Frattini length of $G$.  Thus, for groups of bounded Frattini length, the
order-scale bound $p^n$ is replaced by the quadratic bound $O_p(n^2)$.
For $E_m=(\Z/p\Z)^m$, we further prove the sharp estimate
$\tau_p(E_m)\asymp_p m^2$.

\end{abstract}

\maketitle

\section{Introduction}\label{sec:introduction}
The existence of unramified extensions with prescribed Galois group has been studied by many authors.  Fr\"ohlich showed that every finite group occurs in this way \cite{Frohlich}, but the resulting base field generally has very large degree.  A series of works on alternating groups obtained unramified extensions over quadratic fields \cite{EGM,Kedlaya,Kondo,Uchida,Yamamoto}. For a finite $p$-group $G$, Nomura constructed an unramified extension $M/F$ with $\Gal(M/F)\simeq G$ and $F/\Q$ elementary abelian \cite{Nomura}.  This substantially reduces the degree, but does not require $M$ to be the maximal unramified $p$-extension of $F$.

Let $\Lp(F)$ denote the maximal unramified pro-$p$ extension of $F$. Ozaki proved that every finite $p$-group $G$ can be realized as $\Gal(\Lp(F)/F)$ \cite{Ozaki}.  Thus maximality is retained, although the original theorem is qualitative and produces fields of very large degree.  Hajir, Maire and Ramakrishna made Ozaki's theorem effective \cite{HMR-Ozaki}: if $p\nmid h(k_0)$ and $\mu_p\not\subset k_0$, they construct infinitely many $F/k_0$ such that
\[
 \Gal(\Lp(F)/F)\simeq G,
 \qquad
 [F:k_0]\leq p^2|G|.
\]
Taking $k_0=\Q$ gives a bound of order $|G|$.  The purpose of this article is to retain Ozaki's maximality condition while reducing this degree by using the Frattini structure of $G$.\\\\
Define
\[
 \tau_p(G)=
 \min\left\{
 [F:\Q]:\Gal(\Lp(F)/F)\simeq G
 \right\}.
\]
For a finite $p$-group $G$, put $\PhiG^0(G)=G$ and
\[
 \PhiG^{i+1}(G)=\PhiG(\PhiG^i(G)),
 \qquad
 \PhiG(H)=H^p[H,H].
\]
The Frattini length of $G$ is
\[
 \ell_\Phi(G)=\min\{\ell:\PhiG^\ell(G)=1\}.
\]
Our main result is the following.

\begin{theorem}\label{thm:main}
Let $p$ be an odd prime and let $G$ be a finite $p$-group of order
$p^n>1$.  Then
\begin{equation}\label{eq:main-bound}
 \tau_p(G)
 \leq
 p^{\ell_\Phi(G)+
 \left\lceil\log_p\left(\binom{n+2}{2}+1\right)\right\rceil}.
\end{equation}
In particular,
\begin{equation}\label{eq:min-bound}
 \tau_p(G)\ll_p
 \min\left\{
 |G|,
 p^{\ell_\Phi(G)}(1+\log_p|G|)^2
 \right\}.
\end{equation}
\end{theorem}

Writing $|G|=p^n$, the earlier effective bound is $O_p(p^n)$, whereas
Theorem~\ref{thm:main} gives
\[
 O_p\bigl(p^{\ell_\Phi(G)}(1+n)^2\bigr).
\]
Hence bounded Frattini length changes an exponential bound into a
quadratic one.  For example, $\tau_p(H^r)\ll_{p,H}r^2$ for fixed $H$.

For elementary abelian groups the quadratic scale is sharp:
Theorem~\ref{thm:elementary} proves
$\tau_p((\Z/p\Z)^m)\asymp_p m^2$.\\

The paper is organized as follows.  Sections~\ref{sec:preliminaries}--
\ref{sec:ray} provide the cohomological and ray-class ingredients.
Section~\ref{sec:batch} proves the main batch proposition, and
Section~\ref{sec:induction} applies it along the Frattini series.
Section~\ref{sec:elementary} treats the elementary abelian case.

\section{Some preliminary results}\label{sec:preliminaries}

Throughout the paper, $p$ is an odd prime.  Cohomology of a finite
$p$-group is taken with continuous cochains, which in this finite
setting is ordinary group cohomology.  For a finite $p$-group $Q$, put
\[
 h^i(Q,M)=\dim_{\Fp}H^i(Q,M),
 \qquad
 h^i(Q)=h^i(Q,\Fp).
\]
If $M$ is a finite-dimensional left $\Fp[Q]$-module, write
\[
 M^\vee=\Hom_{\Fp}(M,\Fp),
 \qquad
 (q\varphi)(m)=\varphi(q^{-1}m).
\]

\subsection{Two cohomological estimates}

\begin{lemma}\label{lem:cohomology-bounds}
Let $Q$ be a finite $p$-group of order $p^q$.  Then
\[
 h^1(Q)\leq q,
 \qquad
 h^2(Q)\leq\binom{q+1}{2}.
\]
\end{lemma}

\begin{proof}
The first inequality follows from
$h^1(Q)=\dim_{\Fp}Q/\PhiG(Q)$.  Choose a central series
\[
 1=Q_0\subset Q_1\subset\cdots\subset Q_q=Q
\]
with $Q_{j+1}/Q_j\simeq C_p$.  The Lyndon--Hochschild--Serre
five-term sequence gives
\[
 h^1(Q_{j+1})\leq h^1(Q_j)+1,
 \qquad
 h^2(Q_{j+1})\leq h^2(Q_j)+h^1(Q_j)+1;
\]
see also \cite[Proposition~2.4]{HMR-Ozaki}.  Since
$h^1(Q_j)\leq j$, summing the second inequality gives
\[
 h^2(Q)\leq\sum_{j=0}^{q-1}(j+1)
 =\binom{q+1}{2}.
\]
\end{proof}

\subsection{Equivariant Kummer--Chebotarev}

We shall choose primes orbit by orbit.  The following lemma permits us
to prescribe an entire $R$-linear residue-symbol coordinate at one
prime of $k$.

\begin{lemma}\label{lem:equivariant-chebotarev}
Let $K/k$ be a $Q$-extension, where $Q$ is a finite $p$-group.  Assume
that $p$ is unramified in $K/\Q$ and that
$\mu_p\not\subset K$.  Put $R=\Fp[Q]$.  Let
\[
 Z\subset K^\times/K^{\times p}
\]
be a finite-dimensional $R$-submodule, and let
$L:Z\to R$ be $R$-linear.  Outside any prescribed finite set of
primes, there are infinitely many primes $v$ of $k$ with the following
properties:
\begin{enumerate}
\item $v$ splits completely in $K(\mu_p)/k$;
\item $v_p(\Nm v-1)=1$;
\item after choosing a prime $\mathfrak v$ of $K$ above $v$ and a
generator of the local $p$-Sylow subgroup, the residue-symbol map on
the $Q$-orbit of $\mathfrak v$ is $L$.
\end{enumerate}
\end{lemma}

\begin{proof}
Put $F=K(\mu_p)$.  Since $[F:K]$ is prime to $p$, the natural map
\[
 K^\times/K^{\times p}\longrightarrow F^\times/F^{\times p}
\]
is injective.  Kummer theory gives a perfect pairing
\[
 \Gal(F(\sqrt[p]{Z})/F)\times Z\longrightarrow\mu_p.
\]

The group algebra $R$ is symmetric.  More explicitly, restriction to
the coefficient of $1\in Q$ gives an isomorphism
\begin{equation}\label{eq:symmetric-algebra}
 \Hom_R(Z,R)\xrightarrow{\ \sim\ }Z^\vee.
\end{equation}
Its inverse sends $\varphi\in Z^\vee$ to
\begin{equation}\label{eq:Lphi}
 L_\varphi(z)=
 \sum_{g\in Q}\varphi(g^{-1}z)g.
\end{equation}
Thus $L$ determines a Kummer automorphism of
$F(\sqrt[p]{Z})/F$, after fixing an identification
$\mu_p\simeq\Fp$.

We also impose the exact $p$-part of the residue-field order.  Put
\[
 C=K(\mu_{p^2}).
\]
Absolute unramifiedness of $K$ at $p$ gives $[C:F]=p$.  Let
$\Delta=\Gal(F/K)$.  Conjugation by $\Delta$ on
$\Gal(F(\sqrt[p]{Z})/F)$ is through the nontrivial cyclotomic
character, while its action on $\Gal(C/F)$ is trivial.  It follows
that
\begin{equation}\label{eq:kummer-cyclotomic-disjoint}
 F(\sqrt[p]{Z})\cap C=F.
\end{equation}
Indeed, a nontrivial intersection would be $C$ and would give the same
one-dimensional $\Delta$-module with these two different actions.

Choose a nontrivial element of $\Gal(C/F)$ and combine it, using
\eqref{eq:kummer-cyclotomic-disjoint}, with the Kummer automorphism
attached to $L$.  Chebotarev's theorem in the resulting extension over
$k$ gives infinitely many primes $v$ whose Frobenius restricts trivially
to $F$ and has the prescribed two components over $F$.  Hence $v$ splits
completely in $F/k$.

Choose a prime $\mathfrak V$ above $v$ in the compositum and put
$\mathfrak v=\mathfrak V\cap K$.  After choosing a generator of the
local $p$-Sylow subgroup at $\mathfrak v$, let
$\lambda_{\mathfrak v}:Z\to\Fp$ be the corresponding local residue
functional.  The Kummer Frobenius condition gives
$\lambda_{\mathfrak v}=\varphi$, up to a nonzero scalar, and that
scalar is absorbed by the choice of the local generator.  Since
$[K:k]$ is a power of $p$ while $[k(\mu_p):k]$ divides $p-1$, we have
$K\cap k(\mu_p)=k$.  Thus the $Q$-action lifts to $F$ while fixing
$\mu_p$, and transporting the chosen local generator along the
$Q$-orbit gives, for every $g\in Q$,
\[
 \lambda_{g\mathfrak v}(z)
 =\lambda_{\mathfrak v}(g^{-1}z)
 =\varphi(g^{-1}z).
\]
Consequently the residue-symbol map on the whole orbit is
\[
 z\longmapsto
 \sum_{g\in Q}\lambda_{g\mathfrak v}(z)g
 =\sum_{g\in Q}\varphi(g^{-1}z)g
 =L(z)
\]
by \eqref{eq:Lphi}.  This proves the third assertion.

The nontrivial Frobenius in $C/F$ says that
$\Nm v\equiv1\pmod p$ but $\Nm v\not\equiv1\pmod{p^2}$.  This proves
the second assertion.  Chebotarev also permits the exclusion of any
fixed finite set.
\end{proof}

\begin{remark}\label{rem:expanded-radical}
Lemma~\ref{lem:equivariant-chebotarev} applies to any finite
$R$-submodule $Z$ of $K^\times/K^{\times p}$, not only to global
units.  We will first choose some prime orbits, adjoin Kummer classes
whose divisors are supported on those orbits, and then prescribe the
values of the remaining residue maps on the enlarged module $Z$.
\end{remark}

\section{Units and the arithmetic fundamental class}
\label{sec:units}

Let $K/k$ be an everywhere unramified Galois extension with
\[
 Q=\Gal(K/k)
\]
a nontrivial finite $p$-group.  Assume that $p\nmid h(K)$ and
$\mu_p\not\subset K$.  Put
\[
 E_{K,p}=\cO_K^\times\otimes_{\Z}\Z_p,
 \qquad U=E_{K,p}/pE_{K,p},
 \qquad R=\Fp[Q].
\]

\subsection{The third syzygy}

The following consequence of the Ritter--Weiss Tate sequence is the
module-theoretic input of the construction.

\begin{lemma}\label{lem:unit-syzygy}
There is an exact sequence of $R$-modules
\begin{equation}\label{eq:unit-syzygy}
 0\longrightarrow U\longrightarrow P_2\longrightarrow P_1
 \longrightarrow R\longrightarrow\Fp\longrightarrow0,
\end{equation}
where $P_1$ and $P_2$ are finite free $R$-modules.  If
$\eta_{K/k}\in H^3(Q,U)$ is its Yoneda class, then for every finite
$R$-module $W$ the map
\begin{equation}\label{eq:eta-surjection}
 \Hom_R(U,W)\longrightarrow H^3(Q,W),
 \qquad f\longmapsto f_*\eta_{K/k},
\end{equation}
is surjective.  Moreover,
\begin{equation}\label{eq:H3U-one}
 H^3(Q,U)\simeq\Fp.
\end{equation}
\end{lemma}

\begin{proof}
Let $I_{\Z_p[Q]}$ be the augmentation ideal of $\Z_p[Q]$.
Since the $p$-part of $\Cl(K)$ is trivial, the translation module in
the $p$-adic Ritter--Weiss sequence is $I_{\Z_p[Q]}$.  Corollary~2.2
of \cite{KataokaOzaki} gives an exact sequence
\[
 0\longrightarrow E_{K,p}\longrightarrow P
 \longrightarrow F\longrightarrow I_{\Z_p[Q]}
 \longrightarrow0,
\]
where $F$ is free and $P$ is projective.  The ring $\Z_p[Q]$ is local,
so $P$ is free.  Splicing with
\[
 0\longrightarrow I_{\Z_p[Q]}\longrightarrow\Z_p[Q]
 \longrightarrow\Z_p\longrightarrow0
\]
gives
\[
 0\longrightarrow E_{K,p}\longrightarrow P\longrightarrow F
 \longrightarrow\Z_p[Q]\longrightarrow\Z_p\longrightarrow0.
\]
All modules occurring as intermediate images are $\Z_p$-torsion-free.
Reduction modulo $p$ is therefore exact and gives
\eqref{eq:unit-syzygy}.

Sequence \eqref{eq:unit-syzygy} is a projective resolution through
degree three.  Applying $\Hom_R(-,W)$ to the last three projective
terms gives
\[
 H^3(Q,W)
 \simeq
 \operatorname{coker}\!\left(
 \Hom_R(P_2,W)\longrightarrow\Hom_R(U,W)
 \right).
\]
Under this identification, the class of
$f\in\Hom_R(U,W)$ is precisely the Yoneda pushout
$f_*\eta_{K/k}$.  Hence the quotient map
$\Hom_R(U,W)\twoheadrightarrow H^3(Q,W)$ is
\eqref{eq:eta-surjection}.

Let $I_Q$ be the augmentation ideal of $R$.  Dimension shifting in
\eqref{eq:unit-syzygy} gives
\[
 H^3(Q,U)\simeq H^1(Q,I_Q).
\]
The exact sequence
\[
 0\longrightarrow I_Q\longrightarrow R\longrightarrow\Fp
 \longrightarrow0
\]
and $H^i(Q,R)=0$ for $i>0$ give
\[
 H^1(Q,I_Q)\simeq
 \Fp/\operatorname{aug}(R^Q).
\]
Now $R^Q=\Fp T_Q$, where $T_Q=\sum_{q\in Q}q$, and
$\operatorname{aug}(T_Q)=|Q|=0$ in $\Fp$.  This proves
\eqref{eq:H3U-one}.
\end{proof}

\subsection{Comparison with the fundamental class}

We next identify the class which is relevant to the arithmetic
extension problem.  Let $J_K$ and $C_K=J_K/K^\times$ be the id\`ele
and id\`ele class groups.  Define
\[
 J_K^u=
 \prod_{w<\infty}\cO_{K_w}^\times
 \times\prod_{w\mid\infty}K_w^\times,
\]
viewed as a subgroup of $J_K$, and put
\[
 \cJ=J_K^u\otimes_{\Z}\Z_p,
 \qquad
 \cC=C_K\otimes_{\Z}\Z_p.
\]
These are ordinary algebraic tensor products.

\begin{lemma}\label{lem:arithmetic-class}
There is an exact sequence
\begin{equation}\label{eq:idele-units}
 0\longrightarrow E_{K,p}\longrightarrow\cJ
 \longrightarrow\cC\longrightarrow0.
\end{equation}
The module $\cJ$ is $Q$-cohomologically trivial.  If
$u_{K/k}\in H^2(Q,\cC)$ is the $p$-primary global fundamental class
and
\[
 \theta_{K/k}=\delta(u_{K/k})\in H^3(Q,E_{K,p}),
\]
then the reduction
\[
 \overline\theta_{K/k}\in H^3(Q,U)
\]
is nonzero.  Consequently, for every finite $R$-module $W$, the map
\begin{equation}\label{eq:arithmetic-surjection}
 \Hom_R(U,W)\longrightarrow H^3(Q,W),
 \qquad f\longmapsto f_*\overline\theta_{K/k},
\end{equation}
is surjective.
\end{lemma}

\begin{proof}
The standard id\`ele sequence is
\[
 0\longrightarrow\cO_K^\times\longrightarrow J_K^u
 \longrightarrow C_K\longrightarrow\Cl(K)\longrightarrow0.
\]
The ring $\Z_p$ is flat over $\Z$, and
$\Cl(K)\otimes\Z_p=0$ because $p\nmid h(K)$.  This proves
\eqref{eq:idele-units}.

For a finite place $v$ of $k$ and a place $w$ of $K$ above $v$, the
local extension $K_w/k_v$ is unramified.  Its decomposition group is
cyclic, and the local unit group $\cO_{K_w}^\times$ has trivial Tate
cohomology.  The infinite decomposition groups are trivial because
$K/k$ has odd degree and is unramified at infinity.  Shapiro's lemma,
applied orbit by orbit, shows that $J_K^u$ is $Q$-cohomologically
trivial.  Tensoring with $\Z_p$ preserves this assertion.  Indeed, a
complete resolution of a finite group may be chosen with finitely
generated free terms, and flatness preserves the exactness of the
resulting Tate cochain complex.

The usual global class formation gives
\[
 H^2(Q,C_K)\simeq\Z/|Q|\Z,
\]
with the global fundamental class as a generator; see
\cite[Chapter~VIII, Theorem~4.7]{MilneCFT}.  Since $Q$ is finite, we
may compute cohomology using a resolution by finitely generated free
$\Z[Q]$-modules.  Flatness of $\Z_p$ over $\Z$ therefore gives
\[
 H^2(Q,\cC)
 \simeq H^2(Q,C_K)\otimes_{\Z}\Z_p
 \simeq\Z_p/|Q|\Z_p,
\]
and the image $u_{K/k}$ of the global fundamental class is a generator.
Since $\cJ$ is cohomologically trivial, the boundary map in
\eqref{eq:idele-units} is an isomorphism.  Thus
$\theta_{K/k}$ generates the cyclic group
\[
 H^3(Q,E_{K,p})\simeq\Z_p/|Q|\Z_p.
\]

Because $\mu_p\not\subset K$, the $p$-adic completion of the unit
group has no $p$-torsion; in particular $E_{K,p}$ is a finite free
$\Z_p$-module.  Hence
\[
 0\longrightarrow E_{K,p}\xrightarrow{\,p\,}E_{K,p}
 \longrightarrow U\longrightarrow0
\]
is exact.  The associated cohomology sequence shows that the kernel of
\[
 H^3(Q,E_{K,p})\longrightarrow H^3(Q,U)
\]
is exactly $pH^3(Q,E_{K,p})$.  Therefore the image
$\overline\theta_{K/k}$ of a generator is nonzero.

By Lemma~\ref{lem:unit-syzygy}, $H^3(Q,U)$ is one-dimensional.  The
class $\eta_{K/k}$ is also nonzero, since otherwise the surjection
\eqref{eq:eta-surjection} with $W=U$ would be the zero map.  Hence
$\overline\theta_{K/k}$ and $\eta_{K/k}$ differ by a nonzero scalar,
and \eqref{eq:arithmetic-surjection} follows from
\eqref{eq:eta-surjection}.
\end{proof}

\begin{remark}\label{rem:trivial-Q}
When $Q=1$, all positive-degree cohomology groups in this section
vanish.  In that case the condition involving
$\overline\theta_{K/k}$ is empty; the ray construction below starts
with any required surjection from the unit space.
\end{remark}

\subsection{A surjective unit map}

\begin{proposition}\label{prop:unit-map}
Let
\begin{equation}\label{eq:WBpresentation}
 0\longrightarrow W\xrightarrow{i}V=R^t
 \xrightarrow{\partial}B\longrightarrow0
\end{equation}
be exact, and let $\alpha\in H^2(Q,B)$.  Suppose that
\[
 U\simeq R^\lambda\oplus N
\]
and
\begin{equation}\label{eq:lambda-exact-condition}
 \lambda\geq d_R(W)
 =t-\dim_{\Fp}B_Q+h^1(Q,B^\vee).
\end{equation}
Then there is a surjective $R$-homomorphism
\[
 f:U\longrightarrow W
\]
such that
\begin{equation}\label{eq:class-equation}
 f_*\overline\theta_{K/k}=\delta_\partial(\alpha).
\end{equation}
Here
\[
 \delta_\partial:H^2(Q,B)\xrightarrow{\ \sim\ }H^3(Q,W)
\]
is the connecting isomorphism associated to
\eqref{eq:WBpresentation}.  For $Q=1$, the same conclusion holds with
\eqref{eq:class-equation} omitted.
\end{proposition}

\begin{proof}
Since $V$ is free, $\delta_\partial$ is an isomorphism.  If $Q\neq1$,
Lemma~\ref{lem:arithmetic-class} gives an $R$-homomorphism
$f_0:U\to W$ satisfying \eqref{eq:class-equation}.  If $Q=1$, take
$f_0=0$.

By Lemma~\ref{lem:kernel-generators} and
\eqref{eq:lambda-exact-condition}, there is a surjection
\[
 g:R^\lambda\longrightarrow W.
\]
Define $f$ to be $g$ on $R^\lambda$ and $f_0$ on $N$.  Then $f$ is
surjective.  The difference $f-f_0$ factors through the projective
module $R^\lambda$, and therefore induces the zero map on positive
degree cohomology.  Thus $f$ still satisfies
\eqref{eq:class-equation}.
\end{proof}

We shall obtain the free rank $\lambda$ from the exact formula of
Hajir--Maire--Ramakrishna.

\begin{lemma}\label{lem:Minkowski-rank}
Suppose that $K=\Lp(k)$ is finite over $k$, and put
$Q=\Gal(K/k)$.  If $\mu_p\not\subset k$, then
\[
 U\simeq R^\lambda\oplus N,
 \qquad
 \lambda=r_1(k)+r_2(k)-1+h^1(Q)-h^2(Q)
\]
for some finite $R$-module $N$.
\end{lemma}

\begin{proof}
This is Fact~5 of \cite{HMR-Ozaki}; it is also a consequence of the
unit-module formula in \cite{HMR-Deficiency}.
\end{proof}

\subsection{The number of generators of a relation module}

Let $R=\Fp[Q]$ and let $I_Q$ be its augmentation ideal.  Since $Q$ is
a $p$-group, $R$ is a local ring with Jacobson radical $I_Q$.  Thus the
minimal number of generators of a finite $R$-module $N$ is
\[
 d_R(N)=\dim_{\Fp}N_Q,
 \qquad N_Q=N/I_QN.
\]

\begin{lemma}\label{lem:kernel-generators}
Let
\[
 0\longrightarrow W\longrightarrow R^t
 \xrightarrow{\partial}B\longrightarrow0
\]
be an exact sequence of finite $R$-modules.  Then
\begin{equation}\label{eq:kernel-generators}
 d_R(W)
 =t-\dim_{\Fp}B_Q+h^1(Q,B^\vee).
\end{equation}
If $d=h^1(Q)$ and $b=\dim_{\Fp}B$, then
\[
 d_R(W)\leq t+bd.
\]
\end{lemma}

\begin{proof}
Taking $Q$-coinvariants gives an exact sequence
\[
 0\longrightarrow H_1(Q,B)\longrightarrow W_Q
 \longrightarrow\Fp^t\longrightarrow B_Q\longrightarrow0.
\]
Moreover,
\[
 \dim_{\Fp}H_1(Q,B)=h^1(Q,B^\vee).
\]
This proves \eqref{eq:kernel-generators}.  A $1$-cocycle of $Q$ with
values in $B^\vee$ is determined by its values on a minimal set of
$d$ generators of $Q$.  Hence $h^1(Q,B^\vee)\leq bd$.
\end{proof}

\section{Ray extensions with a prescribed extension class}
\label{sec:ray}

\subsection{The ray class sequence}

Let $K$ be a number field with $p\nmid h(K)$, and let $S$ be a finite
set of finite primes of $K$, all away from $p$.  Suppose that the
$p$-Sylow subgroup of $(\cO_K/\mathfrak p)^\times$ has order $p$ for
every $\mathfrak p\in S$.  Put
\[
 U_K=\cO_K^\times/\cO_K^{\times p},
 \qquad
 V_S=\bigoplus_{\mathfrak p\in S}\Fp e_{\mathfrak p}.
\]
After choosing generators of the local $p$-Sylow subgroups, the ray
class sequence gives
\begin{equation}\label{eq:ray-sequence}
 U_K\xrightarrow{\lambda_S}V_S
 \longrightarrow\Cl_{\cM}(K)(p)\longrightarrow0,
 \qquad
 \cM=\prod_{\mathfrak p\in S}\mathfrak p.
\end{equation}
Here $\Cl_{\cM}(K)(p)$ denotes the $p$-Sylow subgroup of the ray class
group.  In particular, it is elementary abelian.

The following maximality criterion will be used after the local
decomposition groups have been prescribed.

\begin{lemma}\label{lem:ray-maximality}
Let $M/K$ be the full $p$-ray class field for a square-free modulus
$\cM$ supported at primes away from $p$.  Suppose that
\[
 B=\Gal(M/K)
\]
is elementary abelian, that every inertia group has order $p$, and that
\begin{equation}\label{eq:wedge-spanning-prelim}
 \sum_{\mathfrak p\mid\cM}
 \bigwedge^2D_{\mathfrak p}(M/K)=\bigwedge^2B.
\end{equation}
Then $p\nmid h(M)$.
\end{lemma}

\begin{proof}
Suppose that $\Cl(M)(p)\neq0$, and put
\[
 \overline C=\Cl(M)(p)/p\Cl(M)(p).
\]
The finite $p$-group $B$ acts on $\overline C$.  Since $\Fp[B]$ is
local, the coinvariants $\overline C_B$ of this nonzero module are
nonzero.  Choose a nonzero quotient
$\overline C_B\twoheadrightarrow\Fp$.  Equivalently, there is a
$B$-equivariant quotient $\overline C\twoheadrightarrow\Fp$, where
$B$ acts trivially on the target.  By class field theory the kernel
corresponds to an extension $H/M$ such that
\[
 [H:M]=p,
 \qquad H/M\text{ is everywhere unramified}.
\]
The $B$-stability of the kernel implies that $H/K$ is Galois, and the
trivial action on the quotient gives an exact sequence
\[
 1\longrightarrow Z\longrightarrow P\longrightarrow B
 \longrightarrow1,
 \qquad Z\simeq C_p,\quad Z\subseteq Z(P).
\]
Since $B$ is elementary abelian and $Z$ is central, commutators define
an alternating $\Fp$-linear map
\[
 \beta:\bigwedge^2B\longrightarrow Z.
\]

Fix a prime $\mathfrak p\mid\cM$ and a prime of $H$ above it.  Since
$H/M$ is unramified, the inertia subgroup of the corresponding local
decomposition group $\widetilde D_{\mathfrak p}\subseteq P$ still has
order $p$.  It is normal in the $p$-group
$\widetilde D_{\mathfrak p}$ and hence central, because
$|\operatorname{Aut}(C_p)|=p-1$.  The quotient by inertia is cyclic.
It follows that $\widetilde D_{\mathfrak p}$ is abelian.  Therefore
$\beta$ vanishes on
$\bigwedge^2D_{\mathfrak p}(M/K)$.  By
\eqref{eq:wedge-spanning-prelim}, $\beta=0$, and $P$ is abelian.

The extension $H/K$ is unramified outside $\cM$.  At a prime dividing
$\cM$ its inertia group has order $p$ and the ramification is tame, so
the local conductor exponent is one.  Hence the conductor of $H/K$
divides $\cM$.  Since $H/K$ is an abelian $p$-extension, it must be
contained in the full $p$-ray class field $M$, a contradiction.
\end{proof}

\subsection{The extension class of the ray field}

We now compare the cohomology class chosen in
Proposition~\ref{prop:unit-map} with the class of the actual Galois
group of a ray extension.

\begin{proposition}\label{prop:ray-extension-class}
Retain the assumptions and notation of
Proposition~\ref{prop:unit-map}.  Suppose that $t$ prime orbits as in
Lemma~\ref{lem:equivariant-chebotarev} have been chosen so that the
global unit residue map is
\[
 \lambda_S=i\circ f:U\longrightarrow V=R^t.
\]
Let $\cM$ be the product of all primes of $K$ in these orbits, and let
$M/K$ be the full $p$-ray class field of modulus $\cM$.  Then
\[
 \Gal(M/K)\simeq B
\]
as a marked $Q$-module.  Moreover, $M/k$ is Galois and the extension
\[
 1\longrightarrow B\longrightarrow\Gal(M/k)
 \longrightarrow Q\longrightarrow1
\]
has class $\alpha\in H^2(Q,B)$.
\end{proposition}

\begin{proof}
Every selected prime has local multiplicative group with $p$-Sylow
subgroup $C_p$.  Since $p\nmid h(K)$, the $p$-part of the ray class
sequence is
\[
 U\xrightarrow{i f}V\longrightarrow\Cl_{\cM}(K)(p)
 \longrightarrow0.
\]
The map $f$ is onto $W$, so
\[
 \Cl_{\cM}(K)(p)=V/i(W)\simeq B
\]
under the marking induced by $\partial$.  The modulus is $Q$-stable.
Consequently its full $p$-ray class field is Galois over $k$, and the
conjugation action on its Galois group is the prescribed action on
$B$.

It remains to determine the extension class.  We use the reciprocity
convention for which the Shafarevich--Weil theorem sends the global
fundamental class to the class of the Galois group extension.  Local
reciprocity gives a commutative diagram
\begin{equation}\label{eq:idele-ray-diagram}
\begin{array}{ccccccccc}
0&\longrightarrow&E_{K,p}&\longrightarrow&\cJ&\longrightarrow&\cC
&\longrightarrow&0\\
&&\mathllap{\bmod p\ }\downarrow f&&\downarrow\rho
&&\downarrow\operatorname{Art}_M&&\\
0&\longrightarrow&W&\xrightarrow{i}&V&\xrightarrow{\partial}&B
&\longrightarrow&0.
\end{array}
\end{equation}
The middle vertical map is the product of the local residue maps at
the primes dividing $\cM$; outside $\cM$ local units map to zero.  Its
restriction to global units is $i f$, so the left square commutes.  The
right square is global reciprocity.

Let $c\in Z^2(Q,\cC)$ represent the global fundamental class
$u_{K/k}$ and choose a $2$-cochain $j$ with values in $\cJ$ which maps
to $c$.  Then $dj$ takes values in $E_{K,p}$ and represents
$\theta_{K/k}$.  Put $v=\rho(j)$.  Diagram
\eqref{eq:idele-ray-diagram} gives
\begin{equation}\label{eq:cochain-naturality}
 \partial v=\operatorname{Art}_M(c),
 \qquad
 dv=i f(\overline{dj}).
\end{equation}

By the Shafarevich--Weil theorem, the cocycle
$\operatorname{Art}_M(c)$ represents the class
$\alpha_M\in H^2(Q,B)$ of $\Gal(M/k)$; see
\cite[p.~115]{Fesenko}.  The definition of the connecting
homomorphism and \eqref{eq:cochain-naturality} now give
\[
 \delta_\partial(\alpha_M)
 =f_*\overline\theta_{K/k}
 =\delta_\partial(\alpha).
\]
Since $\delta_\partial$ is an isomorphism,
$\alpha_M=\alpha$.
\end{proof}

\section{The main construction}\label{sec:batch}

We now combine the preceding ingredients.  The $Q$-action on the
elementary abelian kernel in the following proposition is arbitrary.

\begin{proposition}[Main batch proposition]\label{prop:batch}
Let $k$ be a totally real number field in which $p$ is unramified.  Suppose that $K=\Lp(k)$ is finite over $k$, and put
\[
 Q=\Gal(K/k),\qquad R=\Fp[Q].
\]
Let
\begin{equation}\label{eq:target-extension}
 1\longrightarrow A\longrightarrow\Gamma
 \longrightarrow Q\longrightarrow1
\end{equation}
be an extension in which $A$ is an elementary abelian $p$-group of
dimension $a\geq1$, with its given $Q$-action.  Set
\[
 B=A\oplus\Fp e,\qquad b=a+1,\qquad
 t=b+\binom b2=\binom{a+2}{2},
\]
where $Q$ acts trivially on $e$.  If the free $R$-rank $\lambda$ of
$\cO_K^\times/\cO_K^{\times p}$ satisfies
\begin{equation}\label{eq:batch-exact-rank}
 \lambda\geq
 t-\dim_{\Fp}B_Q+h^1(Q,B^\vee),
\end{equation}
then there are a cyclic extension $D/k$ of degree $p$ and a field $M$
such that
\[
 M=\Lp(D),
 \qquad
 \Gal(M/D)\simeq\Gamma,
 \qquad
 p\nmid h(M).
\]
The fields $D$ and $M$ are totally real, and $p$ is unramified in
both $D/\Q$ and $M/\Q$.

A sufficient condition for \eqref{eq:batch-exact-rank} is
\begin{equation}\label{eq:batch-simple-rank}
 \lambda\geq t+bd,
 \qquad d=h^1(Q).
\end{equation}
\end{proposition}

\begin{proof}
Let $P=\Gamma\times C_p$ and identify the kernel of $P\to Q$ with
\[
 B=A\oplus\Fp e.
\]
Write
\[
 \alpha\in H^2(Q,B)
\]
for the marked extension class of $P$.  Choose an $\Fp$-basis
$a_1,\ldots,a_a$ of $A$ and put
\begin{equation}\label{eq:transverse-basis}
 c_1=e,
 \qquad
 c_{j+1}=e+a_j\quad(1\leq j\leq a).
\end{equation}
Then $c_1,\ldots,c_b$ is an $\Fp$-basis of $B$.  Moreover, for every
$g\in Q$ and every $j$,
\begin{equation}\label{eq:transversality-lines}
 \Fp(gc_j)\cap A=0,
\end{equation}
because the $e$-coordinate of $gc_j$ is one.

Let
\[
 V=
 \bigoplus_{j=1}^bRx_j
 \ \oplus\!
 \bigoplus_{1\leq i<j\leq b}Ry_{ij}\simeq R^t.
\]
We call the first summands controller coordinates and the second
summands geometry coordinates.  Define an $R$-linear map
\begin{equation}\label{eq:partial-columns}
 \partial:V\longrightarrow B,
 \qquad
 \partial(x_j)=c_j,
 \qquad
 \partial(y_{ij})=c_j.
\end{equation}
The controller columns contain an $\Fp$-basis of $B$, so $\partial$ is
surjective.  Put $W=\ker\partial$.  The rank hypothesis and
Proposition~\ref{prop:unit-map} give a surjection
\[
 f:U\longrightarrow W,
 \qquad U=\cO_K^\times/\cO_K^{\times p},
\]
such that
\begin{equation}\label{eq:batch-class-equation}
 f_*\overline\theta_{K/k}=\delta_\partial(\alpha)
\end{equation}
when $Q\neq1$.  For $Q=1$ the equation is vacuous.  Set
\[
 \Lambda=i\circ f:U\longrightarrow V.
\]

We first choose the geometry primes.  For every $i<j$, apply
Lemma~\ref{lem:equivariant-chebotarev} to the $y_{ij}$-coordinate of
$\Lambda$.  This gives a prime $v_{ij}$ of $k$, split completely in
$K(\mu_p)$, and a selected prime $\mathfrak v_{ij}$ of $K$ above it.
All these primes are distinct, and
\[
 v_p(\Nm v_{ij}-1)=1.
\]

Let $h=h(K)$.  Since $p\nmid h$, for every selected geometry prime
choose $\pi_{ij}\in K^\times$ such that
\begin{equation}\label{eq:pi-principal}
 \mathfrak v_{ij}^{\,h}=(\pi_{ij}).
\end{equation}
Inside $K^\times/K^{\times p}$ we have a direct sum of $R$-modules
\begin{equation}\label{eq:expanded-radical}
 Z=U\ \oplus\!
 \bigoplus_{1\leq i<j\leq b}R[\pi_{ij}].
\end{equation}
Indeed, take the valuation at each prime in the $Q$-orbit of
$\mathfrak v_{ij}$.  A relation in \eqref{eq:expanded-radical} makes
$h$ times every coefficient zero modulo $p$.  Since $p\nmid h$, all
coefficients vanish.  The same calculation shows that every
$R[\pi_{ij}]$ is free of rank one.

We next choose the controller primes.  Let $S_{\rm geom}$ be the
set of all primes of $K$ in the geometry orbits chosen above.  The
choice of a local generator at each such prime determines a local
residue-symbol coordinate on $K^\times/K^{\times p}$.  Via the
identification of the geometry orbits with the corresponding summands
$Ry_{kl}\subset V$ and the map $\partial:V\to B$, each local symbol
therefore has a well-defined image in $B$.

Fix $\mathfrak v=\mathfrak v_{ij}$.  Define
$\gamma_{\mathfrak v}\in B$ to be the sum of the images of the local
residue symbols of $\pi_{ij}$ at all primes
$\mathfrak w\in S_{\rm geom}$ with
$\mathfrak w\neq\mathfrak v$.  Thus the contribution from the
selected prime $\mathfrak v$ itself is omitted; the other primes in
its $Q$-orbit are included.  This element is determined entirely by
the geometry primes, their chosen local generators, and $\pi_{ij}$,
all of which have already been fixed.

For every controller index $r$, extend the $x_r$-coordinate of
$\Lambda$ from $U$ to an $R$-linear map
\[
 L_r:Z\longrightarrow R.
\]
On each generator $\pi_{ij}$ choose the values so that
\begin{equation}\label{eq:controller-equation}
 \sum_{r=1}^bL_r(\pi_{ij})c_r
 =-h c_i-\gamma_{\mathfrak v_{ij}}
 \qquad(1\leq i<j\leq b).
\end{equation}
These equations are soluble because the $c_r$ form an $\Fp$-basis of
$B$; the values $L_r(\pi_{ij})$ may even be chosen in the scalar copy
of $\Fp$ in $R$.  Directness in \eqref{eq:expanded-radical} shows that
the choices for distinct $\pi_{ij}$ are independent.

Apply Lemma~\ref{lem:equivariant-chebotarev} to the maps $L_r$.  We
obtain distinct controller primes $q_1,\ldots,q_b$ of $k$, split
completely in $K(\mu_p)$, with selected primes of $K$ above them, such
that
\[
 v_p(\Nm q_r-1)=1.
\]
Their restrictions to $U$ are the controller coordinates of
$\Lambda$, while their values on the $\pi_{ij}$ satisfy
\eqref{eq:controller-equation}.

Let $S$ be the set of all primes of $K$ above the $t$ chosen primes of
$k$, and let
\[
 \cM=\prod_{\mathfrak p\in S}\mathfrak p.
\]
The global unit residue map is exactly $\Lambda=i f$.  Let $M/K$ be
the full $p$-ray class field of modulus $\cM$.  By
Proposition~\ref{prop:ray-extension-class},
\begin{equation}\label{eq:batch-group-class}
 \Gal(M/K)\simeq B,
 \qquad
 [\Gal(M/k)]=\alpha\in H^2(Q,B).
\end{equation}
Thus, as a marked extension of $Q$ by $B$,
\begin{equation}\label{eq:P-identification}
 \Gal(M/k)\simeq P=\Gamma\times C_p.
\end{equation}

We now determine the local decomposition groups.  At the selected
prime $\mathfrak v_{ij}$, the inertia line is $\Fp c_j$.  Let
\[
 \overline{\Frob}_{\mathfrak v_{ij}}
 \in B/\Fp c_j
\]
be the Frobenius in the quotient by inertia.  Apply global reciprocity
to \eqref{eq:pi-principal}.  The local unit contribution at
$\mathfrak v_{ij}$ lies in the inertia line and hence disappears after
passing to $B/\Fp c_j$.  Modulo that inertia line, global reciprocity
gives
\[
 h\overline{\Frob}_{\mathfrak v_{ij}}
 +\gamma_{\mathfrak v_{ij}}
 +\sum_{r=1}^bL_r(\pi_{ij})c_r=0.
\]
Using \eqref{eq:controller-equation} and $p\nmid h$, we obtain
\begin{equation}\label{eq:geometry-frobenius}
 \overline{\Frob}_{\mathfrak v_{ij}}
 =c_i\pmod{\Fp c_j}.
\end{equation}
Consequently,
\begin{equation}\label{eq:geometry-plane}
 D_{\mathfrak v_{ij}}(M/K)=\langle c_i,c_j\rangle.
\end{equation}
The planes in \eqref{eq:geometry-plane} give
\begin{equation}\label{eq:wedge-spanning}
 \sum_{\mathfrak p\mid\cM}
 \bigwedge^2D_{\mathfrak p}(M/K)=\bigwedge^2B.
\end{equation}
Lemma~\ref{lem:ray-maximality} now shows that
\begin{equation}\label{eq:class-number-M}
 p\nmid h(M).
\end{equation}

Under \eqref{eq:P-identification}, regard $\Gamma$ as
$\Gamma\times1$.  Then $\Gamma\cap B=A$.  The inertia lines over a
selected prime are the lines $\Fp(gc_j)$, for $g\in Q$.  By
\eqref{eq:transversality-lines}, each of them has trivial intersection
with $\Gamma$.  Let
\[
 D=M^\Gamma.
\]
It follows that $D/k$ is cyclic of degree $p$ and $M/D$ is everywhere
unramified, with
\[
 \Gal(M/D)\simeq\Gamma.
\]
If $\Lp(D)$ were larger than $M$, the nontrivial pro-$p$ group
$\Gal(\Lp(D)/M)$ would have a quotient of order $p$.  This would give
an unramified $C_p$-extension of $M$, contradicting
\eqref{eq:class-number-M}.  Hence $M=\Lp(D)$.

All ramification introduced in the construction is finite and away
from $p$.  Moreover, $M/k$ is a Galois extension of odd degree.  Since
$k$ is totally real, $M$ and its subfield $D$ are totally real.  They
remain unramified at the primes above $p$.

Finally, \eqref{eq:batch-simple-rank} implies
\eqref{eq:batch-exact-rank} by
Lemma~\ref{lem:kernel-generators}.
\end{proof}

\section{Proof of the main theorem}
\label{sec:induction}

We first construct a starting field with many units and trivial
$p$-class group.

\begin{lemma}\label{lem:starting-field}
For every integer $r\geq0$, there is a totally real cyclic extension
$k_0/\Q$ of degree $p^r$ which is unramified at $p$ and satisfies
$p\nmid h(k_0)$.
\end{lemma}

\begin{proof}
For $r=0$, take $k_0=\Q$.  Suppose that $r\geq1$.  By Dirichlet's
theorem, choose a rational prime $\ell$ such that
\[
 v_p(\ell-1)=r.
\]
Let $k_0$ be the unique subfield of degree $p^r$ in
$\Q(\mu_\ell)$.  Since $p^r$ is odd, complex conjugation fixes $k_0$,
so $k_0$ is totally real.  The prime $\ell$ is the only ramified
rational prime and is totally ramified; in particular, $k_0$ is
unramified at $p$.

Consider the chain of degree-$p$ subextensions from $\Q$ to $k_0$.
At every step exactly one finite tame prime ramifies.  Fact~3 of
\cite{HMR-Ozaki} shows inductively that the class number at every level
is prime to $p$.  In particular, $p\nmid h(k_0)$.
\end{proof}

\begin{proof}[Proof of Theorem~\ref{thm:main}]
Put
\[
 C_n=\binom{n+2}{2}+1,
 \qquad
 r=\left\lceil\log_p C_n\right\rceil.
\]
Choose the field $k_0$ of Lemma~\ref{lem:starting-field}.  Then
\begin{equation}\label{eq:k0-degree}
 [k_0:\Q]=p^r\geq C_n,
 \qquad p\nmid h(k_0).
\end{equation}

Let
\[
 G=\PhiG^0(G)\supset\PhiG^1(G)\supset\cdots
 \supset\PhiG^\ell(G)=1,
 \qquad \ell=\ell_\Phi(G),
\]
and define
\[
 Q_i=G/\PhiG^i(G)
 \qquad(0\leq i\leq\ell).
\]
Thus $Q_0=1$ and $Q_\ell=G$.  For $0\leq i<\ell$, there is an exact
sequence
\begin{equation}\label{eq:Frattini-layer}
 1\longrightarrow A_i\longrightarrow Q_{i+1}
 \longrightarrow Q_i\longrightarrow1,
 \qquad
 A_i=\PhiG^i(G)/\PhiG^{i+1}(G).
\end{equation}
The kernel $A_i$ is elementary abelian, with the conjugation action of
$Q_i$.

We construct fields $k_i$ inductively so that
\begin{equation}\label{eq:induction-invariant}
 \Gal(\Lp(k_i)/k_i)\simeq Q_i,
\end{equation}
$k_i$ is totally real and unramified at $p$, and
\begin{equation}\label{eq:induction-degree}
 [k_i:\Q]=p^{r+i}.
\end{equation}
For $i=0$, these assertions follow from
\eqref{eq:k0-degree}, since $p\nmid h(k_0)$ implies
$\Lp(k_0)=k_0$.

Suppose that $k_i$ has been constructed, and put
\[
 K_i=\Lp(k_i),
 \qquad Q=Q_i,
 \qquad |Q|=p^q,
 \qquad a=\dim_{\Fp}A_i.
\]
Then $p\nmid h(K_i)$.  Let $d=h^1(Q)$ and let $\lambda_i$ be the free
$\Fp[Q]$-rank of
$\cO_{K_i}^\times/\cO_{K_i}^{\times p}$.  Since $k_i$ is totally real,
Lemma~\ref{lem:Minkowski-rank} gives
\begin{equation}\label{eq:lambda-stage}
 \lambda_i=[k_i:\Q]-1+d-h^2(Q).
\end{equation}

We verify the sufficient rank condition
\eqref{eq:batch-simple-rank}.  The order of $Q_{i+1}$ gives
\[
 q+a=\log_p|Q_{i+1}|\leq n.
\]
By Lemma~\ref{lem:cohomology-bounds}, $d\leq q$ and
$h^2(Q)\leq\binom{q+1}{2}$.  Hence
\begin{align}
 1+\binom{a+2}{2}+ad+h^2(Q)
 &\leq
 1+\binom{a+2}{2}+aq+\binom{q+1}{2}\notag\\
 &=\binom{a+q+1}{2}+a+2\notag\\
 &\leq\binom{n+2}{2}+1=C_n.
 \label{eq:quadratic-budget}
\end{align}
Combining \eqref{eq:k0-degree}, \eqref{eq:lambda-stage} and
\eqref{eq:quadratic-budget}, we obtain
\[
 \lambda_i
 \geq\binom{a+2}{2}+(a+1)d.
\]

Apply Proposition~\ref{prop:batch} to
\eqref{eq:Frattini-layer}.  It produces a cyclic extension
$k_{i+1}/k_i$ of degree $p$ such that
\[
 \Gal(\Lp(k_{i+1})/k_{i+1})\simeq Q_{i+1}.
\]
The proposition also preserves total reality and unramifiedness at
$p$.  This completes the induction.

At $i=\ell$ we obtain a field $F=k_\ell$ with
\[
 \Gal(\Lp(F)/F)\simeq G,
 \qquad
 [F:\Q]=p^{r+\ell}.
\]
This is \eqref{eq:main-bound}.  Since
\[
 p^{\lceil\log_p C_n\rceil}<pC_n,
\]
we also have
\[
 \tau_p(G)<p^{\ell_\Phi(G)+1}
 \left(\binom{n+2}{2}+1\right).
\]
The Hajir--Maire--Ramakrishna bound
$\tau_p(G)\leq p^2|G|$ gives \eqref{eq:min-bound}.
\end{proof}

\begin{corollary}\label{cor:bounded-length}
Fix an odd prime $p$ and an integer $L$.  If $G$ ranges over finite
$p$-groups with $\ell_\Phi(G)\leq L$, then
\[
 \tau_p(G)\ll_{p,L}(1+\log|G|)^2.
\]
In particular, for every fixed finite $p$-group $H$,
\[
 \tau_p(H^s)\ll_{p,H}s^2.
\]
\end{corollary}

\begin{proof}
The first assertion is immediate from Theorem~\ref{thm:main}.  For the
second, use
\[
 \PhiG^i(H^s)=\PhiG^i(H)^s,
\]
so $\ell_\Phi(H^s)=\ell_\Phi(H)$, while
$\log_p|H^s|=s\log_p|H|$.
\end{proof}

\begin{corollary}\label{cor:extraspecial}
If $G$ is an extraspecial $p$-group of exponent $p$, then
\[
 \tau_p(G)\ll_p(1+\log_p|G|)^2.
\]
\end{corollary}

\begin{proof}
For such a group, $\PhiG(G)=Z(G)$ has order $p$ and
$\PhiG^2(G)=1$.  Thus $\ell_\Phi(G)=2$.
\end{proof}

\section{Application: elementary abelian
\texorpdfstring{$p$}{p}-groups}\label{sec:elementary}

We finish by proving the sharper elementary abelian estimate and its
optimality.

\begin{theorem}\label{thm:elementary}
Let $p$ be an odd prime and $E_m=(\Z/p\Z)^m$.  Then
\[
 \tau_p(E_m)\asymp_p m^2.
\]
More precisely,
\begin{equation}\label{eq:elementary-upper}
 \tau_p(E_m)
 \leq
 p^{1+\left\lceil
 \log_p\left(1+\binom{m+1}{2}\right)
 \right\rceil}.
\end{equation}
\end{theorem}

\begin{proof}[Proof of Theorem~\ref{thm:elementary}]
Put
\[
 A=E_m,
 \qquad b=m+1,
 \qquad g=\binom b2,
 \qquad
 r=\left\lceil\log_p(g+1)\right\rceil.
\]
Choose $k_0$ as in Lemma~\ref{lem:starting-field}.  Since $k_0$ is
totally real,
\[
 \dim_{\Fp}\cO_{k_0}^\times/\cO_{k_0}^{\times p}
 =p^r-1\geq g.
\]

Apply Proposition~\ref{prop:batch} with $Q=1$ and $\Gamma=A$.  In the
presentation used in its proof,
\[
 t=b+\binom b2,
 \qquad
 d_R(W)=t-\dim_{\Fp}B=g.
\]
Thus the exact rank condition \eqref{eq:batch-exact-rank} is satisfied.
The proposition gives a degree-$p$ extension $D/k_0$ such that
\[
 \Gal(\Lp(D)/D)\simeq E_m.
\]
Consequently,
\[
 \tau_p(E_m)\leq[D:\Q]
 =p^{1+\lceil\log_p(g+1)\rceil},
\]
which is \eqref{eq:elementary-upper}.

For the lower bound, the mod-$p$ cohomology ring of $E_m$ gives
\[
 h^1(E_m)=m,
 \qquad
 h^2(E_m)=m+\binom m2.
\]
Hence
\[
 \operatorname{Def}(E_m)
 =h^2(E_m)-h^1(E_m)=\binom m2.
\]
If $E_m\simeq\Gal(\Lp(F)/F)$, the Shafarevich--Koch relation bound
gives
\[
 \operatorname{Def}(E_m)
 \leq\dim_{\Fp}
 \cO_F^\times/\cO_F^{\times p}
 \leq[F:\Q];
\]
see \cite{HMR-Deficiency}.  Therefore
\[
 \tau_p(E_m)\geq\binom m2.
\]
Together with \eqref{eq:elementary-upper}, this proves
$\tau_p(E_m)\asymp_p m^2$.
\end{proof}
\section*{Use of AI-assisted tools}
ChatGPT (GPT-5.6, OpenAI) was used to assist with the exploration of cohomological arguments and the verification of certain calculations. The author independently verified all mathematical content.

\end{document}